\documentclass[11pt]{article}

\let\oldmarginpar\marginpar
\renewcommand\marginpar[1]{\-\oldmarginpar[\raggedleft\footnotesize #1]
{\raggedright\footnotesize #1}}

\usepackage{mathptmx}
\usepackage{amsmath}
\usepackage{amscd}
\usepackage{amssymb}
\usepackage{amsthm}
\usepackage{xspace}
\usepackage[all,tips]{xy}
\usepackage[dvips]{graphicx}
\usepackage{verbatim}
\usepackage{syntonly}
\usepackage{hyperref}

\theoremstyle{plain}
\newtheorem{thm}{Theorem}[section]
\newtheorem{cor}[thm]{Corollary}

\newtheorem{lemma}[thm]{Lemma}

\theoremstyle{definition}

\DeclareMathOperator{\SL}{SL} \DeclareMathOperator{\PSL}{PSL}

 \DeclareMathOperator{\Mat}{M}

\DeclareMathOperator{\Ram}{Ram}
\DeclareMathOperator{\Br}{Br}

\DeclareMathOperator{\Inv}{Inv}
\DeclareMathOperator{\Res}{Res}

\newcommand{\bdef}{\overset{\text{def}}{=}}

\newcommand{\nid}{\noindent}

\newcommand{\tss}{\smallskip\smallskip}

\newcommand{\iny}{\infty}

\newcommand{\abs}[1]{\left\vert#1\right\vert}
\newcommand{\set}[1]{\left\{#1\right\}}
\newcommand{\brac}[1]{\left[#1\right]}
\newcommand{\pr}[1]{\left( #1 \right) }
\newcommand{\su}{\subset}

\newcommand{\lra}{\longrightarrow}

\newcommand{\B}[1]{\ensuremath{\mathbf{#1}}}
\newcommand{\BB}[1]{\ensuremath{\mathbb{#1}}}
\newcommand{\Cal}[1]{\ensuremath{\mathcal{#1}}}

\newcommand{\Hy}{\ensuremath{\B{H}}}

\newcommand{\Q}{\ensuremath{\B{Q}}}
\newcommand{\R}{\ensuremath{\B{R}}}
\newcommand{\Z}{\ensuremath{\B{Z}}}
\newcommand{\C}{\ensuremath{\B{C}}}

\usepackage{fancyhdr}

\fancypagestyle{plain}

\begin{document}


\title{\textbf{Geometric spectra and commensurability}}
\author{D. B. McReynolds\thanks{The work was partial supported by an NSF grant}}
\maketitle

\begin{abstract}
\nid The work of Reid, Chinburg--Hamilton--Long--Reid, Prasad--Rapinchuk, and the author with Reid have demonstrated that geodesics or totally geodesic submanifolds can sometimes be used to determine the commensurability class of an arithmetic manifold. The main results of this article show that generalizations of these results to other arithmetic manifolds will require a wide range of data. Specifically, we prove that certain incommensurable arithmetic manifolds arising from the semisimple Lie groups of the form $(\SL(d,\R))^r \times (\SL(d,\C))^s$ have the same commensurability classes of totally geodesic submanifolds coming from a fixed field. This construction is algebraic and shows the failure of determining, in general, a central simple algebra from subalgebras over a fixed field. This, in turn, can be viewed in terms of forms of $\SL_d$ and the failure of determining the form via certain classes of algebraic subgroups.
\end{abstract}

\section{Introduction}

\nid The present article addresses the following general geometric question. \tss

\nid \textbf{Question 1.}
\emph{How much of the geometry of a Riemmanian manifold $M$ is encoded in the geometry of the totally geodesic submanifolds of $M$?}\tss

\nid This question was the focus of the recent article \cite{McR} where two general results were shown for hyperbolic 3--manifolds. First, two arithmetic hyperbolic 3--manifolds $M_1,M_2$ with the same totally geodesic surfaces, up to commensurability, are commensurable provided they have a single totally geodesic surface. Second, given any finite volume hyperbolic 3--manifold $M$, there exist finite covers $M_1,M_2$ of $M$ with precisely the same totally geodesic surfaces (counted with multiplicity).\tss 

\nid The article \cite{McR} was motivated by analogous results in spectral geometry where the focuses are the geodesic length spectrum of $M$ and the spectrum of the Laplace--Beltrami operator acting on $L^2(M)$. Reid \cite{Reid} proved that if $X,Y$ are Riemann surfaces with the same geodesic length spectrum and $X$ is arithmetic, then $X,Y$ are commensurable. In particular, $Y$ is also arithmetic. Via Selberg's trace formula, one gets an identical result for the eigenvalue spectrum of the Laplace--Beltrami operator acting on $L^2(X)$. We refer to these results as commensurability rigidity since the spectral data is sufficient for determining the commensurability class of the manifold. More recently, Chinburg--Hamilton--Long--Reid \cite{CHLR} proved that the geodesic length spectrum for arithmetic hyperbolic 3--manifolds also enjoys the same commensurability rigidity. Specifically, two arithmetic hyperbolic 3--manifolds with the same geodesic length spectrum are commensurable. Prasad--Rapinchuk \cite{PR} have determined for a large class of locally symmetric manifolds when commensurability rigidity holds. It is worth noting that it does not always hold; see also \cite{Linowitz} and \cite{LSV}. \tss

\nid These results all have algebraic analogs that are the primary tools in the above rigidity results. For instance, the main algebraic observation in \cite{McR} was a similar result about quaternion algebras. Let $A_1,A_2$ be quaternion algebras over number fields $K_j$ with subfield $F \su K_j$ with $[K_j:F]=2$. We can associate to $A_1,A_2$ the sets of quaternion algebras $B$ over $F$ such that $A_j = B \otimes_F K_j$. In \cite{McR}, we proved that these sets determine the algebras provided they are non-empty. Namely, if these sets are equal and non-empty, then $A_1 \cong A_2$. In particular, the fields $K_1,K_2$ are isomorphic. Similarly, Reid \cite{Reid}, in proving commensurability rigidity for the geodesic length spectrum, proved that the splitting fields of the invariant quaternion algebra determine the invariant quaternion algebra. See \cite{CRR}, \cite{GS}, and \cite{Meyer} for results describing the extent of the failure of this statement for general algebras and also generalizations of these rigidity results. \tss

\nid The following algebraic questions also serves as motivation presently.\tss

\nid \textbf{Question 2.} \emph{How much of the structure of a central simple algebra $A$ is encoded in the subalgebras of $A$? How much of the structure of an algebraic group $\B{G}$ is encoded in the algebraic subgroups of $\B{G}$?}\tss

\nid The work of Prasad--Rapinchuk addressed the second question for maximal subtori of almost absolutely simple algebraic groups $\B{G}$. The work in \cite{McR} focused on certain $\SL_2$--forms over real fields for certain algebraic forms of $\SL_2$.\tss

\nid This article continues this theme by demonstrating via example that generalizations of the above rigidity results need a large range of geometric (or algebraic) data. Recall that for an extension $K/F$, we have a map on Brauer groups (see Section 2 for the definitions)
\[ \Res_{K/F}\colon \Br(F) \lra \Br(K) \]
given by
\[ \Res_{K/F}(B) = B \otimes_F K. \]
With this notation set, we have the following result.

\begin{thm}\label{MainQuat}
There exist infinitely many pairs of number fields $K,K'$ and infinitely many pairs of central simple division algebras $A,A'$ over $K,K'$, respectively such that
\[ (\Res_{K/\Q})^{-1}(A) = (\Res_{K'/\Q})^{-1}(A') \ne \emptyset. \]
There exist infinitely many distinct pairs $A,A'$ for a fixed degree $d$ and pairs $A,A'$ for every degree $d\geq 2$.
\end{thm}

\nid Our next theorem shows that $(\Res_{K/\Q})^{-1}(A)$ and $(\Res_{K'/\Q})^{-1}(A')$ can nearly be equal without being equal.

\begin{thm}\label{AlmostEqual}
There exists number fields $K,K'$ and central simple algebras $A,A'$ over $K,K'$ such that 
\[ \abs{(\Res_{K/\Q})^{-1}(A) \cap (\Res_{K'/\Q})^{-1}(A')} = \infty \] 
but
\[ (\Res_{K/\Q})^{-1}(A) \ne (\Res_{K'\Q})^{-1}(A'). \]
\end{thm}

\nid The failure in equality in the above theorem is quite mild, only involving the local behavior of the algebras $B/\Q$ at the ramified places of $\Q$ in the extensions $K,K'$. \tss

\nid The algebras $A,A'$ produce arithmetic lattices in semisimple Lie groups via orders and restriction of scalars to $\Q$. Specifically, the semisimple Lie groups are
\[ G_{r,s} = (\SL(d,\R))^r \times (\SL(d,\C))^s \] 
for certain pairs of $r,s$. The simplest example is when $d=2$ and we can take the algebras to be
\[ A = \Mat(2,K), \quad A' = \Mat(2,K'). \]
The numbers $r,s$ correspond in this case to the number of real and complex places of $K,K'$. The lattices can be taken to be $\SL(2,\Cal{O}_K),\SL(2,\Cal{O}_{K'})$, and the associated arithmetic orbifolds are
\[ M = ((\Hy^2)^r \times (\Hy^3)^s)/\PSL(2,\Cal{O}_K), \quad M' = ((\Hy^2)^r \times (\Hy^3)^s)/\PSL(2,\Cal{O}_{K'}). \]
Here, $\Hy^2,\Hy^3$ are real hyperbolic 2 and 3--space. These orbifolds are sometimes called Hilbert--Blumenthal modular varieties. By construction, they have the same commensurability classes of totally geodesic surfaces coming from the field $\Q$. Each commensurability class of surfaces is associated to a $\Q$--quaternion algebra $B$ such that $B \otimes_\Q \R \cong \Mat(2,\R)$, $B \otimes_\Q K \cong \Mat(2,K)$, and $B \otimes_\Q K' \cong \Mat(2,K')$.  \tss 

\nid More generally, for algebras $A,A'$ in Theorem \ref{MainQuat}, we have associated manifolds $M_A,M_{A'}$ given by $X_{r,s}/\Cal{O}$, $X_{r,s}/\Cal{O}'$, where $X_{r,s}$ is the symmetric space associated to $G_{r,s}$ and $\Cal{O},\Cal{O}'$ are orders in $A,A'$. The manifolds have the property that a manifold $N_B$ coming from a central simple $\Q$--algebra $B$ of degree $d$ arises as an arithmetic totally geodesic submanifold of $M_A$ if and only if it arises as an arithmetic, totally geodesic submanifold of $M_{A'}$, up to the commensurability of $N_B$. In particular, these manifolds have a rich class of totally geodesic submanifolds that are unable to determine the commensurability class of the manifold. This construction works for infinitely many distinct pairs $(r,s)$ and produces infinitely many distinct pairs of commensurability classes of manifolds for each pair $(r,s)$. We obtain the following geometric corollary from the above discussion.\tss

\begin{cor}
Let $A,A'$ be central simple algebras over $K,K'$ with 
\[ (\Res_{K/\Q})^{-1}(A) = (\Res_{K'/\Q})^{-1}(A') \]
and $M_A,M_{A'}$, associated arithmetic manifolds for $A,A'$. Let $B$ is a central simple algebra defined over a $\Q$ of the same degree as $A,A'$ and $N_B$ is an associated arithmetic manifold for $B$. Then $N_B$ arises as a totally geodesic submanifold of $M_A$ up to commensurability if and only if $N_B$ arises as a totally geodesic submanifold of $M_{A'}$ up to commensurability.
\end{cor}

\nid On the level of algebraic groups, Theorem \ref{MainQuat} can be restated as:

\begin{cor}\label{MainForms}
There exist number fields $K,K'$ and $K,K'$--forms $\B{G}/K,\B{G}'/K'$ of $\SL_d$ with precisely the same sets 
\[ \set{\B{H}/\Q~:~\B{H}(K) \cong \B{G}(K)} \ne \emptyset \]
and
\[ \set{\B{H}/\Q~:~\B{H}(K') \cong \B{G}'(K')}. \]
There are infinitely many pairs $K,K'$ and for each pair and each $d>1$, infinitely many groups $\B{G},\B{G}'$ satisfying the above conditions.
\end{cor}

\nid One can state this in terms of Galois cohomology and maps between Galois cohomology sets; below we work with Brauer groups instead of Galois cohomology explicitly though they are one in the same (see \cite[Chapter 14]{Pierce} or \cite[Chapter X]{Serre}). The groups $\B{G},\B{G}'$ in the theorem are called $K$,$K'$--forms for the groups $\B{H}$. In particular, $\B{G},\B{G}'$ are the same $K,K'$--forms for $\Q$--forms of $\SL_d$. In \cite{CMM}, we will investigate at more depth relationships in Galois cohomology associated various constructions over pairs of fields. Finally, this article is far from exhaustive on the types of constructions possible via the methods presented here. We remark in the final section more on generalizations of the constructions in this article.

\paragraph{Acknowledgements}

\nid A debt is owed to Amir Mohammadi for conversations on this material that lead me to address the questions in this article. In addition, I thank Ted Chinburg, Britain Cox, Jordan Ellenberg, Skip Garibaldi, Ben Linowitz, Jeff Meyer, Nicholas Miller, Alan Reid, Matthew Stover, and Henry Wilton for conversations on the material in this article.

\section{Preliminaries}

\nid This section contains some preliminary material required in the sequel.

\subsection{Number fields}

\nid By a number field $K$, we mean a finite extension of $\Q$. We denote the set of places of $K$ by $\Cal{P}_K$. Each place $\omega$ resides over a unique place $q \in \Cal{P}_\Q$ and we write $\omega \mid q$ when $\omega$ is a place over $q$. For $q=\iny$, the places $\omega$ are just the real and complex places and are often referred to as the archimedean places. The associated extension $K_\omega/\Q_q$ of local fields has degree given by $e(K_\omega/\Q_q)f(K_\omega/\Q_q)$ where $e(K_\omega/\Q_q)$ is the ramification degree and $f(K_\omega/\Q_q)$ is called the inertial degree (see \cite[p. 19, Proposition 3]{CF}). There are only finitely many primes $q$ for which $e(K_\omega/\Q_q)>1$ for some $\omega$ (see \cite[p. 22, Corollary 2]{CF}). In addition, it is well known (see \cite[p. 65, Theorem 21]{Marcus}) that
\begin{equation}\label{FieldDegree}
\sum_{\omega \mid q} e(K_\omega/\Q_q)f(K_\omega/\Q_q) = \deg(K/\Q).
\end{equation}
The number of distinct places over a fixed $q$ will be denoted by $g_q(K/\Q)$.

\subsection{Central simple algebras and Brauer groups}

\nid We refer the reader to \cite{Pierce} for a general introduction to central simple $K$--algebras. The Morita equivalence classes of central simple algebras over a number field $K$ with tensor product form a group called the Brauer group of $K$. We denote the Brauer group of $K$ by $\Br(K)$ (see \cite[12.5]{Pierce}). Given any extension of fields $L/K$, we obtain a homomorphism
\[ \Res_{L/K}\colon \Br(K) \lra \Br(L) \]
via
\[ \Res_{L/K}([B]) = [B \otimes_K L]. \]
For $[A] \in \Br(L)$, we denote the fiber of $\Res_{L/K}$ over $[A]$ by $(\Res_{L/K})^{-1}([A])$. In fact, we will abuse notation and drop the notation for the Morita equivalence class since we work (almost always) with division algebras and there is a unique division algebra in each Morita equivalence class (see \cite[Proposition b, p. 228]{Pierce}). \tss

\nid For each place $\omega \in \Cal{P}_K$, we have an associated algebra $A_\omega = A \otimes_K K_\omega$. Via local class field theory, we have an isomorphism (see \cite[p. 193, Proposition 6]{Serre})
\[ \Br(K_\omega) \lra \Q/\Z \]
when $\omega$ is a finite place. For a real or complex place, classically via Wedderburn, we have
\[ \Br(\R) = \Z/2\Z, \quad \Br(\C) = 1. \]
From these isomorphisms, for $B \in \Br(K)$ and each place $\omega \in \Cal{P}_K$, we obtain
\[ \Inv_\omega(B) \bdef \Inv(B_\omega) \in \Q/\Z,\frac{1}{2}\Z/\Z,\text{ or } \Z/\Z \]
called the local invariant of $B$ at $\omega$. The total package
\[ \Inv(B) = \set{\Inv_\omega(B)~:~\omega \in \Cal{P}_K} \]
is called the invariant of $B$. By the Albert--Hasse--Brauer--Noether Theorem (see \cite[Section 18.4]{Pierce}), $B$ is determined as a $K$--algebra by $\Inv(B)$. Moreover, any set 
\[  \set{\alpha_\omega \in \Br(K_\omega)~:~\omega \in \Cal{P}_K} \su \prod_{\omega \in \Cal{P}_K} \Br(K_\omega) \] 
can be realized as the invariants of an algebra provided two conditions are met:
\begin{itemize}
\item[(a)]
$\alpha_\omega = 0$ for all but finitely many places $\omega \in \Cal{P}_K$ (see \cite[p. 358, Proposition]{Pierce}).
\item[(b)]
\[ \sum_{\omega \in \Cal{P}_K} \alpha_\omega = 0 \mod \Z. \]
For this condition, see \cite[p. 363, Proposition b]{Pierce}.
\end{itemize}
If $L_\nu/K_\omega$ is a finite extension and $B_\omega \in \Br(K_\omega)$, then (see \cite[p. 193, Proposition 7]{Serre})
\begin{equation}\label{LocalInvariants}
\Inv_\nu(B_\omega \otimes_{K_\omega} L_\nu) = [L_\nu:K_\omega]\Inv_\omega(B_\omega).
\end{equation}
We say $A/K$ is unramified at a place $\omega$ when $A \otimes_K K_\omega \cong \Mat(d,K_\omega)$ and ramified otherwise. In particular, when $\Inv_\omega(A) \ne 0$, $A$ is ramified at $\omega$. We denote the set of places where $A$ is ramified by $\Ram(A)$.

\section{Arithmetically and locally equivalent fields}

\nid We say two number fields $K,K'$ are arithmetically equivalent if $\zeta_K(s) = \zeta_{K'}(s)$. We say $K,K'$ are locally equivalent if $\B{A}_K \cong \B{A}_{K'}$, where
\[ \B{A}_K = \prod_{\omega \in \Cal{P}_K} K_\omega \]
is the ring of $K$--ad\`eles. By work of Iwasawa (see \cite{Komatsu} or \cite{Perlis}), when $K,K'$ are locally equivalent, $K,K'$ are arithmetically equivalent. When $\B{A}_K \cong \B{A}_{K'}$, there is a bijective map (see \cite[Lemma 3]{Komatsu})
\[ \Phi\colon \Cal{P}_K \lra \Cal{P}_{K'} \]
such that for all $\omega \in \Cal{P}_K$, we have $K_\omega \cong K_{\Phi(\omega)}$. Moreover, we have
\[ f(K_\omega/\Q_q) = f(K'_{\Phi(\omega)}/\Q_q), \quad e(K_\omega/\Q_q) = e(K'_{\Phi(\omega)}/\Q_q), \quad [K_\omega:F_\nu] = [K'_{\Phi(\omega)}:F_\nu]. \]
For arithmetically equivalent fields $K,K'$, we have a bijective map (see \cite[Theorem 1]{Perlis})
\[ \Phi\colon \Cal{P}_K \lra \Cal{P}_{K'} \]
such that
\[ f(K_\omega/\Q_q) = f(K_{\Phi(\omega)}'/\Q_q) \]
for any place $\omega \in \Cal{P}_K$. For any unramified prime $q$, we see that
\[ [K_\omega:\Q_q]  = f(K_\omega/\Q_q) = f(K_{\Phi(\omega)}'/\Q_q)  = [K_{\Phi(\omega)}':\Q_q] \]
since $K,K'$ have the same set of unramified primes and at such places we have by definition
\[ e(K_\omega/\Q_q) = e(K_{\Phi(\omega)}'/\Q_q) = 1. \]

\section{Fibers of the restriction map}

\nid Given a central simple $K$--algebra $A$ and number fields $K/F$, if $B \in (\Res_{K/F})^{-1}(A)$, then by definition
\[ A = B \otimes_F K. \]
At any place $\omega \mid \nu$, we have the local equation (\ref{LocalInvariants}) given by
\[ \Inv_\omega(A) = \Inv_\omega(B \otimes_F K)= [K_\omega:F_\nu]\Inv_\nu(B). \]
Notice that for each place $\nu \in \Cal{P}_F$, we have an equation for each $\omega$ over $\nu$. Solving for $\Inv_\nu(B)$, we see that
\[ \frac{\Inv_{\omega_1}(A)}{[K_{\omega_1}:F_\nu]} = \frac{\Inv_{\omega_2}(A)}{[K_{\omega_2}:F_\nu]} = \dots = \frac{\Inv_{\omega_{g_\nu(K/F}}(A)}{[K_{\omega_{g_\nu(K/F)}}:F_\nu]} = \Inv_\nu(B)  \]
holds for $\omega_1,\dots,\omega_{g_\nu(K/F)}$ over $\nu$. In particular, the local invariants of $A$ at all of the places over $\nu$ satisfy
\[ \frac{[K_{\omega_j}:F_\nu]}{[K_{\omega_1}:F_\nu]} \Inv_{\omega_1}(A)= \Inv_{\omega_j}(A), \quad j \in \set{1,\dots,g_\nu(K/F)}. \]
Typically, these equalities will not be satisfied for an algebra $A$.\tss

\nid One of the main results of \cite{McR} was a proof that for an arithmetic hyperbolic 3--manifold with a totally geodesic surface, the invariant quaternion algebra of the 3--manifold is determined (among all invariant quaternion algebras of arithmetic hyperbolic 3--manifolds) by the quaternion algebras over the maximal totally real subfield of the invariant trace field. Specifically, each 3--manifold $M_1,M_2$ has an associated number field $K_1,K_2$ with precisely one complex place called the invariant trace field. Since the manifolds contain a totally geodesic surface, there is a common totally real subfield $F = K_1 \cap K_2$ with $[K_j:F]=2$. The 3--manifolds also each have an associated quaternion algebra $A_1,A_1$ over $K_1,K_2$ called the invariant quaternion algebra. The condition that the manifolds have the same totally geodesic surfaces, up to commensurability, implies that
\begin{equation}\label{3man}
(\Res_{K_1/F})^{-1}(A_1) = (\Res_{K_2/F})^{-1}(A_2)
\end{equation}
The classification of totally geodesic surfaces in this setting also yields that
\begin{equation}\label{3man2}
\abs{(\Res_{K_1/F})^{-1}(A_1)} = \iny.
\end{equation}
Combining (\ref{3man}), (\ref{3man2}) with basic class field theory, we obtain $K_1 \cong K_2$ and thus $A_1 \cong A_2$. For arithmetic manifolds, we get immediately that $M_1,M_2$ are commensurable. \tss 

\nid The following results shows that such rigidity behavior is not always the case.

\begin{thm}\label{AlgebraFlexibility}
There exists infinitely pairs of fields $K_j,K_j'$ over $\Q$ such that the following holds.
\begin{itemize}
\item[(a)]
There exists infinitely many pairs of central simple algebras $A_{i,j}/K_j,A'_{i,j}/K_j'$ such that
\[ (\Res_{K_j/\Q})^{-1}(A_{i,j}) = (\Res_{K_j'/\Q})^{-1}(A'_{i,j}) \ne \emptyset. \]
\item[(b)]
For each pair of fields $K_j,K_j'$, we can take the algebras to have degree $d$ for any $d>1$. In particular, for each pair of fields, there exists infinitely many pairs of algebras for every degree.
\item[(c)]
If $d=2$, the algebras $A_{i,j},A_{i,j}'$ can be constructed so that
\[ \abs{(\Res_{K_j/\Q})^{-1}(A_{i,j})} = \iny. \]
\end{itemize}
\end{thm}

\begin{proof}
Let $K_j,K_j'$ be distinct, locally equivalent number fields. We take here the explicit examples given by \cite[Theorem, p.1]{Komatsu2} which have $\deg(K_j) = \deg(K_j') = 2^j$ for all $j>2$. For simplicity, we set $K_j = K$ and $K_j' = K'$. Let $q_1,\dots,q_r$ be a finite number of primes. We will assume over each prime $q_j$, there is a place $\omega_j$ with inertial degree $1$. Infinitely many primes $q$ have this property by the Cebotarev Density Theorem. We further insist that the primes $q_j$ are also unramified in $K$ (or equivalently $K'$). Since there are only finitely many ramified primes, the set of $r$--tuples of unramified primes $(q_1,\dots,q_r)$ such that for each $q_j$ there is a place $\omega_j$ of $K$ over $q_j$ with inertial degree $1$ is infinite. We get an infinite set for each $r$ but must insist that $d \mid r$. \tss

\nid Let $\set{q_1,\dots,q_r}$ be an $r$--tuple satisfying the above conditions. For each $q_j$, we pick a place $\omega_j \in \Cal{P}_K$ over $q_j$ with inertial degree 1. We define a central simple algebra $A$ over $K$ of degree $d$ by local invariants as follows:
\begin{itemize}
\item[(a)]
For each $j$, over the inert place $\omega_j \mid q_j$, we define the invariant to be
\[ \Inv_{\omega_j}(A) = \frac{1}{d}. \]
\item[(b)]
For each $\omega \mid q_j$, we define the invariant to be
\[ \Inv_\omega(A)=\frac{[K_\omega:\Q_{q_j}]}{[K_{\omega_j}:\Q_{q_j}]} \Inv_{\omega_j}(A)=[K_\omega:\Q_{q_j}]\Inv_{\omega_j}(A). \]
\item[(c)]
Finally, for any place $\omega$ not over one of the $q_j$, we define the invariant to be
\[ \Inv_\omega(A) = 0. \]
\end{itemize}
Via our bijection $\Phi\colon \Cal{P}_K \to \Cal{P}_{K'}$, we define $A'/K'$ via the local data
\[ \Inv_{\omega}(A') = \Inv_{\Phi^{-1}(\omega)}(A) \]
for any place $\omega \in \Cal{P}_{K'}$. By construction of $A,A'$, we have
\[ \abs{\Ram(A)},\abs{\Ram(A')} < \infty. \]
It remains to check that
\[ \sum_\omega \Inv_\omega A = \sum_\omega \Inv_{\Phi(\omega)}(A') = 0 \mod \Z. \]
If this equation holds, then we know that there exist algebras $A,A'$ with the above local invariants. Now, to prove the above sum is zero, we have
\begin{align*}
\sum_{\omega} \Inv_{\omega}(A) &= \sum_{j=1}^r \brac{\sum_{\omega \mid q_j}\pr{ \frac{[K_\omega:\Q_{q_j}]}{d}}} \\
&= \frac{1}{d}\pr{ \sum_{j=1}^r \brac{\sum_{\omega \mid q_j} [K_\omega:\Q_{q_j}]}} \\
&= \frac{1}{d}\brac{ \sum_{j=1}^r m_{q_j,K} }
\end{align*}
where
\[ m_{q_j,K} = \sum_{\omega \mid q_j} [K_\omega:\Q_{q_j}]. \]
However, by (\ref{FieldDegree}), we see that $m_{q_j,K} = [K:\Q]$. In particular,
\[  \frac{1}{d}\brac{ \sum_{j=1}^r m_{q_j,K} } =  \frac{1}{d}\brac{ \sum_{j=1}^r [K:\Q] } = \frac{r[K:\Q]}{d}. \]
By selection, $d \mid r$ and hence
\[ \sum_{\omega \in \Cal{P}_K} \Inv_\omega(A) = \frac{r[K:\Q]}{d} =  0 \mod \Z, \]
as needed. Thus, we know that there exist algebras $A,A'$ that satisfy (a), (b), and (c). \tss 

\nid Next, we argue that
\[ (\Res_{K/\Q})^{-1}(A) = (\Res_{K'/\Q})^{-1}(A'). \]
Given an algebra $B \in (\Res_{K/\Q})^{-1}(A)$, we know that by (\ref{LocalInvariants})
\[ \Inv_\omega(A) = [K_\omega:\Q_q]\Inv_q(B) \]
for every prime $q$ and every place $\omega$ over $q$. Via our bijection $\Phi\colon \Cal{P}_K \to \Cal{P}_{K'}$ we have
\[ [K_\omega:\Q_q] = [K_{\Phi(\omega)}:\Q_q] \]
and
\[ \Inv_{\Phi(\omega)}(A') = \Inv_\omega(A). \]
In tandem, we see that
\begin{align*} 
\Inv_{\Phi(\omega)}(A') &= \Inv_\omega(A) \\
&= [K_\omega:\Q_q]\Inv_q(B) \\
&= [K'_{\Phi(\omega)}:\Q_q]\Inv_q(B).
\end{align*}
The algebra $B \otimes_\Q K'$ has local invariants given by (\ref{LocalInvariants}) and thus
\[ \Inv_{\Phi(\omega)}(B \otimes_\Q K') = [K_{\Phi(\omega)}':\Q_q] \Inv_q(B) = \Inv_{\Phi(\omega)}(A'). \]
Therefore, by the Albert--Hasse--Brauer--Noether Theorem, $A' \cong B \otimes_\Q K'$. In particular, $B \in (\Res_{K'/\Q})^{-1}(A')$ and so
\[ (\Res_{K/\Q})^{-1}(A) \su (\Res_{K'/\Q})^{-1}(A'). \]
The reverse inclusion
\[ (\Res_{K'/\Q})^{-1}(A') \su (\Res_{K/\Q})^{-1}(A) \]
follows from an identical argument.\tss

\nid Finally, to see that $(\Res_{K/\Q})^{-1}(A)$ is non-empty, simply note that by (a) and (b) and the fact that the inertial degree of $K$ at $\omega_j$ is 1, we can define
\[ \Inv_{q_j}(B) = \frac{1}{d}. \]
This is consistent with the local equations
\[ \Inv_\omega(A) = [K_\omega:\Q_{q_j}]\Inv_{q_j}(B) \]
by (b). At any other prime $q$, we could declare $\Inv_q(B)=0$ and thus complete the local data for an algebra $B$ over $\Q$. For this claim, we are again using $d \mid r$ to obtain
\[ \sum_q \Inv_q(B) = 0 \mod \Z. \]
In total, we see that this local data satisfies the necessary conditions to be the local invariants of an algebra $B$ over $\Q$. \tss

\nid However, it can sometimes be the case that we can choose non-zero invariants for $B$ at other place $q$ not on the list $\set{q_1,\dots,q_r}$. Specifically, when $q$ is a place such that for every place $\omega$ over $q$ we have $d \mid [K_\omega:\Q_q]$, then we can certainly set $\Inv_q(B)=1/d$ without violating the local equations
\[ 0 = \Inv_\omega(A) = [K_\omega:\Q_q]\Inv_q(B) = \frac{[K_\omega:\Q_q]}{d} = 0 \mod \Z. \]
In the case $d=2$, we assert that the set $(\Res_{K/\Q})^{-1}(A)$ is infinite. To that end, recall that $\deg(K) = \deg(K') = 2^j$ for $j \geq 3$ and let $\Cal{P}_{2,\Q}$ denote the set of primes $q$ such that for every $\omega \mid q$, we have $2 \mid f(K_\omega/\Q_q)$. By the Cebotarev Density Theorem, the set $\Cal{P}_{2,\Q}$ is infinite. Let $B$ be a $\Q$--algebra such that over the primes $q_1,\dots,q_r$, we have
\[ \Inv_{q_j}(B) = \frac{1}{d}. \] 
Next, set
\[ \Ram(B) = \set{q_1,\dots,q_r} \cup \set{q_1'\dots,q_{r'}'} \]
where $q_j' \in \Cal{P}_{2,\Q}$, $d \mid r'$, and
\[ \Inv_{q_j'}(B) = \frac{1}{d}. \]
Then these invariants satisfy the necessary conditions to be the local invariants for an algebra $B$ over $\Q$. To see that $B \in (\Res_{K/\Q})^{-1}(A)$, we split our consideration into three places. First, for any place $q \in\set{q_1,\dots,q_r}$, we saw from above that $B$ is consistent with the local equation (\ref{LocalInvariants}) for each $\omega \mid q$. Second, for any place $q \notin \Ram(B)$, the algebra $B$ is trivially consistent with the local equation (\ref{LocalInvariants}) since
\[ \Inv_\omega(A) = \Inv_q(B) = 0 \]
for each $\omega \mid q$. Finally, for $q \in \set{q'_1,\dots,q'_{r'}}$, the algebra $B$ satisfies the local equation (\ref{LocalInvariants}) since for each $\omega \mid q$, we have
\[ \Inv_q(B) = \frac{1}{d}, \quad \Inv_\omega(A) = 0 \]
and
\[ d \mid [K_\omega:\Q_q]. \]
Consequently,
\[ \Inv_\omega(B \otimes_\Q K) = \frac{[K_\omega:\Q_q]}{d} = 0 \mod \Z \]
for each $\omega \mid q$. Varying over the finite subsets of $\Cal{P}_{2,\Q}$, we produce infinitely many distinct algebras $B$ in $(\Res_{K/\Q})^{-1}(A)$.\tss

\nid In total, there exists infinitely many choices for the starting pair of fields $K,K'$. For each such pair and each integer $d>1$, there are infinitely many pairs of algebras $A,A'$ of degree $d$ satisfying the conclusions of the theorem. These algebras are given by varying the set of primes $\set{q_1,\dots,q_r}$ where each $q_j$ has a place over it with inertial degree $1$ and each prime $q_j$ is unramified. Finally, we can select $r$ to be any integer with $d \mid r$.
\end{proof}

\nid Theorem \ref{AlgebraFlexibility} visibly implies Theorem \ref{MainQuat}. Also notice that the above argument proves the following.

\begin{cor}\label{LocalMain}
Let $K/K'$ be locally equivalent fields. Then for each $d>1$, there exists central simple division algebras $A,A'$ over $K,K'$ of degree $d$ such that for every subfield $F \su K \cap K'$, we have
\[ (\Res_{K/F})^{-1}(A) = (\Res_{K'/F})^{-1}(A') \ne \emptyset. \]
\end{cor}

\nid We now prove Theorem \ref{AlmostEqual}.

\begin{proof}[Proof of Theorem \ref{AlmostEqual}]
Via \cite[p. 351]{Perlis}, there exist degree 8 non-isomorphic arithmetically equivalent extensions $K,K'$ whose Galois closure is degree 32 where the associated Galois group is a  $2$--group. These fields are also not locally equivalent as the ramified prime 2 has decompositions  given by
\[ (2)_K = Q_1^2Q_2^2Q_3^2Q_4^2, \quad (2)_{K'} = P_1P_2P_3^2P_4^4. \]
The inertial degree is 1 over each of these primes. Now let $A$ be a quaternion algebra over $K$ constructed as in the proof of Theorem \ref{AlgebraFlexibility}. Since $K,K'$ are arithmetically equivalent, we have a bijection of places that preserves inertial degree. Since in the proof of Theorem \ref{AlgebraFlexibility}, we only worked over unramified primes, we can define an associated algebra $A'$ over $K'$ using the proof of Theorem \ref{AlgebraFlexibility}. By the Cebotarev Density Theorem, we know that there are infinitely many primes $q$ where the inertial degree of any place $\omega$ over $q$ is at least 2. Over these primes, we certainly have $2 \mid [K_\omega:\Q_q]$. In particular, we can ramify an algebra $B/\Q$ at these places and still maintain the local equations
\[ 0 = \Inv_\omega(A) = [K_\omega:\Q_q]\Inv_q(B) = \frac{[K_\omega:\Q_q]}{2} = 0 \mod \Z. \]
Thus, we again see that $(\Res_{K/\Q})^{-1}(A)$ is infinite. Moreover, if $2 \notin \Ram(B)$, then $B \in (\Res_{K/\Q})^{-1}(A)$  if and only if $B \in (\Res_{K'/\Q})^{-1}(A')$. Consequently, 
\[ \abs{(\Res_{K/\Q})^{-1}(A) \cap (\Res_{K'/\Q})^{-1}(A')} = \iny. \]
To see that
\[ (\Res_{K/\Q})^{-1}(A) \ne  (\Res_{K'\Q})^{-1}(A'), \]
simply note that over the prime 2, we can ramify the algebra $B$ at $2$ and maintain the local equation (\ref{LocalInvariants}) for $A$ but not for $A'$. Indeed, over each of the four places over $2$ for $K$, we have
\[ [K_{\omega_1}/\Q_2] = [K_{\omega_2}/\Q_2] = [K_{\omega_3}/\Q_2] = [K_{\omega_4}/\Q_2] =  2, \]
while for $K'$, we have
\[ [K'_{\omega_1'}/\Q_2] = [K'_{\omega_2'}/\Q_2] = 1, \quad [K'_{\omega_3'}/\Q_2]=2, \quad [K'_{\omega_4'}/\Q_2] = 4. \]
Since $A,A'$ are unramified at all places over $2$, we see by (\ref{LocalInvariants}) that
\[ \Inv_{\omega_j}(A) = [K_{\omega_j}:\Q_2]\Inv_2(B) = 2\Inv_2(B) = 0 \mod \Z \]
and
\[ \Inv_{\omega_1'}(A') = \Inv_2(B) = 0 \mod \Z. \]
In the first case, clearly we can select $\Inv_2(B)$ to be either $0$ or $1/2$, while in the second case it can only be $0$. This actually provides infinite many quaternion algebras in $(\Res_{K/\Q})^{-1}(A)$ that are not in $(\Res_{K'/\Q})^{-1}(A')$.
\end{proof}

\nid For contrast, we observe the following trivial result which shows the need for working over pairs of fields $K,K'$.

\begin{lemma}\label{AlgebraRigidity}
Let $A,A'$ be central simple $K$--algebras of degree $d$ such that
\[ (\Res_{K/\Q})^{-1}(A) \cap (\Res_{K/\Q})^{-1}(A') \ne \emptyset. \]
Then $A \cong A'$.
\end{lemma}

\begin{proof}
We have $B/\Q$ with $B \in (\Res_{K/\Q})^{-1}(A) \cap (\Res_{K/\Q})^{-1}(A')$ such that
\[ A_1 \cong B \otimes_\Q K \cong A_2. \]
\end{proof}

\nid This trivial lemma is meant to highlight the typical setting. Namely, the difficult work in commensurability rigidity results is often proving that the associated algebraic structures giving rise to the arithmetic manifolds have the same field of definition; see \cite{CHLR}, \cite{McR} and \cite{PR}.

\section{A geometric application}

\nid We refer the reader to \cite{PLR} for the basic background on algebraic groups, forms of algebraic groups, and arithmetic lattices in these algebraic groups. The book of Witte-Morris \cite{Witte} is also provides an excellent introduction to the topic.\tss

\nid Given a number field $K/\Q$ with $r$ real places and $s$ complex places, up to complex conjugation, and a central simple $K$--algebra of degree $d$, we have the associated groups $(A \otimes_K \overline{\tau(K)})^{-1}$, where $\tau$ is one of the above real or complex places. According to Wedderburn's Structure Theorem, we know that
\[ (A\otimes_K \overline{\tau(K)})^{-1} \cong \begin{cases} \SL(d,\C), & \tau \text{ is a complex place} \\ \SL(d,\R),& \tau \text{ is a real place and }(2,d)=1 \\ \SL(d,\R),& \tau \text{ is a real place, } 2\mid d,\text{ }A \text{ splits over }\overline{\tau(K)}\\ \SL(d/2,\BB{H})& \tau \text{ is a real place, } 2\mid d,\text{ }A \text{ does not split over }\overline{\tau(K)}. \end{cases} \] 
For simplicity, we will assume that the fourth case does not happen for any real place. Any $\Cal{O}_K$--order $\Cal{O}$ in $A$ provides, by Borel--Harish-Chandra \cite{BHC}, a lattice $\Cal{O}^1$ in the group
\[ \prod_{\tau \in \Cal{P}_{K,\iny}} (A \otimes_K \overline{\tau(K)})^1 \cong (\SL(d,\R))^{r} \times (\SL(d,\C))^{s}. \]
Any field $F\su K$ and a central simple $F$--algebra $B$ of degree $d$ with $A \cong B \otimes_F K$ provides us with a subgroup of $\Cal{O}^1$ via
\[ (\Cal{O} \cap B)^1 = \Cal{O}_B^1. \]
This subgroup produces an arithmetic lattice in
\[ \prod_{\theta \in \Cal{P}_{F,\iny}} (B \otimes_F \overline{\theta(F)})^1 \]
where $\theta$ ranges over the real and complex places (up to complex conjugation) of $F$. Taking all the possible algebras $B$ over all the subfields of $F \su K$ produces, up to commensurability, all the arithmetic subgroups of $\Cal{O}^1$ arising from groups of type $(\SL(d,\R))^{r'} \times (\SL(d,\C))^{s'}$ where $r' < r$ and $s' < s$. As $F=\Q$ always happens, we can always try to produce submanifolds from $\Q$. Note that even if we have an algebra $B$ over $\Q$ with $A = B \otimes_\Q K$, the algebra $B$ only produces an arithmetic subgroup when $B$ is unramified at the real place of $\Q$. If $d$ is odd, this is always the case. Even when $d$ is not odd, we can ensure that $B$ is unramified at the archimedean place provided the algebra $A$ is unramified at every archimedean place.\tss

\nid One can use the previous section to produce manifolds with universal cover isometric to the symmetric spaces $X_{r,s}$ of $G_{r,s} = (\SL(d,\R))^r \times (\SL(d,\C))^s$ for various pairs $r,s$ and $d$. These manifolds $M_A,M_{A'}$ will have some common totally geodesic submanifolds, up to commensurability, in certain dimensions. The algebras $A,A'$ in the previous section are, by construction, unramified at all archimedean places. In this case, $r,s$ are the number of real and complex places of $K,K'$. \tss

\nid The algebra $A$ only determines $M_A$, up to commensurability. Our construction gives a relationship between the commensurability classes of certain submanifolds of any manifold commensurable to $M_A$ with any manifold commensurable to $M_{A'}$.\tss

\nid For the fields $K,K'$ and a subfield $F \su K$ or $K'$ not contained in $K \cap K'$, the field $F$ could produce totally geodesic submanifolds in the associated manifolds for the algebras $A$ or $A'$. However, these potential submanifolds cannot be immersed as totally geodesic submanifolds in both classes of manifolds. In particular, though these manifolds share a large class of totally geodesic submanifolds, they do not in general contain the same classes of totally geodesic submanifolds up to commensurability. This remarks prompts the following question.\tss

\nid \textbf{Question 3.} \emph{Do there exist incommensurable arithmetic manifolds $M_1,M_2$ with the same class of totally geodesic submanifolds, up to commensurability? Or have precisely the same totally geodesic submanifolds (with or without multiplicity), up to free homotopy?}\tss

\nid To avoid trivialities, we must insist that $M_1,M_2$ do have at least one totally geodesic submanifold beyond geodesics or flats; the work of Prasad--Rapinchuk \cite{PR} provides examples without assuming the manifolds have submanifolds beyond geodesics and flats. \tss

\section{Final remarks}

\nid One can generalize the above construction in a few different ways. First, by Komatsu \cite{Komatsu3}, there exist for any $r$, fields $K_1,\dots,K_r$ that are pairwise non-isomorphic and locally equivalent. In particular, we can produce arbitrarily large collections of algebras $A_1,\dots,A_r$ over $K_1,\dots,K_r$ that pairwise have the same fiber $(\Res_{K_j/\Q})^{-1}(A_j)$. In addition, using relative versions of arithmetic or local equivalence, we can produce examples over larger base fields than $\Q$. These constructions yield examples of manifolds that share an even richer collection of submanifolds coming from subfields of $F = K_1 \cap K_2$. In addition, taking locally equivalent fields $L_1,L_2$ with degrees divisible by other primes, we can for any prime degree produce algebras $A_1,A_2$ over $L_1,L_2$ with
\[ (\Res_{L_1/\Q})^{-1}(A_1) = (\Res_{L_2/\Q})^{-1}(A_2), \quad \abs{(\Res_{L_1/\Q})^{-1}(A_1)} = \iny. \]

\nid We suspect also that these methods work equally as well for other simple non-compact Lie groups and their associated symmetric spaces. In a forthcoming paper with Britain Cox, Benjamin Linowitz, and Nicholas Miller \cite{CMM}, we will explore these generalizations and relations in Galois cohomology. In particular, we provide a more general picture of the work in this article.\tss

\nid After completing this paper, the author learned of the work of Manny Aka \cite{Aka}. Our construction here is a generalization of his construction. Both articles makes essential use of locally equivalent fields. In \cite{CMM}, we will provide a lengthy discussion on the relation of Aka's work and the work here and in \cite{CMM}.



\nid Purdue University, West Lafayette IN 47906. \verb"dmcreyno@math.purdue.edu"


\begin{thebibliography}{9999}

\bibitem{Aka}
M.~Aka, \emph{Arithmetic groups with isomorphic finite quotients},  J. of Algebra \textbf{352} (2012), 322--340

\bibitem{BHC}
A.~Borel and Harish-Chandra, \emph{Arithmetic subgroups of algebraic groups}, Ann. of Math. (2) \textbf{75} (1962), 485--535.

\bibitem{CF}
J.~W.~S.~Cassels and A.~Fr\"{o}lich, \emph{Algebraic number theory}, London Mathematical Society, 1967.

\bibitem{CRR}
V.~I.~Chernousov, A.~S.~Rapinchuk, and I.~A.~Rapinchuk, \emph{On the genus of a division algebra}, C. R. Math. Acad. Sci. Paris
\textbf{350} (2012), 17--18.

\bibitem{CHLR}
T.~Chinburg, E.~Hamilton, D.~D.~Long, and A.~W.~Reid, \emph{Geodesics and commensurability classes of arithmetic hyperbolic 3-manifolds}, Duke Math. J. \textbf{145} (2008), 25--44.

\bibitem{CMM}
B.~Cox, B.~Linowitz, D.~B.~McReynolds, and N.~Miller, \emph{Local invariants and Galois cohomology}, in preparation.

\bibitem{GS}
S.~Garibaldi and D.~Saltman, \emph{Quaternion algebras with the same subfields}, Quadratic forms, linear algebraic groups, and cohomology: Developments in Mathematics, \textbf{18} Springer, (2010), 225--238.

\bibitem{Komatsu}
K.~Komatsu, \emph{On the ad\`ele rings of algebraic number fields}, Kodai Math. Sem. Rep. \textbf{28} (1976), no. 1, 78--84. 

\bibitem{Komatsu2}
K.~Komatsu, \emph{On the ad\`ele rings and zeta-functions of algebraic number fields}, Kodai Math. J. \textbf{1} (1978), no. 3, 394--400. 

\bibitem{Komatsu3}
K.~Komatsu, \emph{On ad\`ele rings of arithmetically equivalent fields}, Acta Arith. \textbf{43} (1984), 93--95. 

\bibitem{Linowitz}
B.~Linowitz, \emph{Isospectral towers of Riemannian Manifolds}, New York J. Math. \textbf{18} (2012) 451--461.

\bibitem{LSV}
A.~Lubotzky,~B.~Samuels, and U.~Vishne, \emph{Division Algebras and Non-Commensurable Isospectral Manifolds}, Duke Math. J. \textbf{135} (2006), 361--379.

\bibitem{Marcus}
D.~A.~Marcus, \emph{Number fields}, Springer--Verlag, 1977.

\bibitem{McR}
D.~B.~McReynolds and A.~W.~Reid, \emph{The genus spectrum of hyperbolic $3$--manifolds}, to appear in Math. Res. Lett.

\bibitem{Meyer}
J.~Meyer, \emph{Division algebras with infinite genus}, preprint.

\bibitem{Perlis}
R.~Perlis, \emph{On the equation {$\zeta \sb{K}(s)=\zeta \sb{K'}(s)$}}, J. Number Theory \textbf{9} (1977), 342--360.

\bibitem{Pierce}
R.~S.~Pierce, \emph{Associative algebras}, Springer--Verlag, 1982. 

\bibitem{PLR}
V.~Platonov and A.~Rapinchuk, \emph{Algebraic groups and number theory}, Academic Press, 1994.

\bibitem{PR}
G.~Prasad and A.~Rapinchuk, \emph{Weakly commensurable arithmetic groups and isospectral locally symmetric spaces}, Publ. Math., Inst. Hautes \'{E}tud. Sci. \textbf{109} (2009), 113--184.

\bibitem{Reid}
A.~W. Reid, \emph{Isospectrality and commensurability of arithmetic hyperbolic $2$- and $3$-manifolds},  Duke Math. J. \textbf{65} (1992), 215--228.

\bibitem{Serre}
J.-P.~Serre, \emph{Local fields}, Springer--Verlag, 1979.

\bibitem{Witte}
D.~Witte-Morris, \emph{Introduction to Arithmetic Groups}.

\end{thebibliography}
\end{document}